\documentclass[11pt,twoside]{preprint}
\usepackage{comment}
\usepackage{hyperref}
\usepackage{breakurl}
\usepackage{times}
\usepackage{microtype}

\usepackage{amsmath}
\makeatletter
\let\over\@@over
\let\atop\@@atop
\makeatother
\usepackage{amssymb}
\usepackage{mathrsfs}
\usepackage{mhequ}
\usepackage{mhenvs}
\usepackage{mhsymb}

\usepackage{enumitem}

\definecolor{darkgreen}{rgb}{0.1,0.7,0.1}
\definecolor{darkred}{rgb}{0.7,0.1,0.1}
\hypersetup{
citecolor=blue,
colorlinks=true, 
urlcolor= blue, 
linkcolor= black, 
bookmarksopen=true, 
pdftitle={PAM R2}, 
}

\newop{diag}

\newcommand\minus{%
  \setbox0=\hbox{-}%
  \vcenter{%
    \hrule width\wd0 height \the\fontdimen8\textfont3%
  }%
}

\def\${|\!|\!|}

\newcommand{\pa}{\mathrm{p}_a}
\newcommand{\e}{\mathrm{e}}

\def\qed{ \hfill $\square$}

\newcommand{\bbE}{\mathbb{E}}

\newcommand{\cB}{\mathcal{B}}
\newcommand{\cC}{\mathcal{C}}

\newcommand{\cE}{\mathcal{E}}

\newcommand{\cM}{\mathcal{M}}

\newcommand{\cO}{\mathcal{O}}

\begin{document}

\title{A simple construction of the continuum \\ parabolic Anderson model on $\R^2$}
\author{Martin Hairer$^1$ and Cyril Labb\'e$^2$}
\institute{University of Warwick, \email{M.Hairer@Warwick.ac.uk}
\and University of Warwick, \email{C.Labbe@Warwick.ac.uk}}

\date{\today}

\maketitle

\begin{abstract}
We propose a simple construction of the solution to the continuum parabolic Anderson model on $\R^2$ 
which does not rely on any elaborate arguments and makes extensive use of the linearity of 
the equation. A logarithmic renormalisation is required to counterbalance the divergent product 
appearing in the equation. Furthermore, we use time-dependent weights in our spaces of distributions 
in order to construct the solution on the unbounded space $\R^2$.
\end{abstract}

\section{Introduction}

The goal of this note is to construct solutions to the continuous parabolic Anderson model:
\begin{equ}[e:PAM]
\tag{PAM}
\partial_t u = \Delta u + u\cdot\xi \;,\qquad u(0,x) = u_0(x)\;.
\end{equ}
Here, $u$ is a function of $t\geq 0$ and $x\in\R^2$, while $\xi$ is a white noise on $\R^2$. Notice that $\xi$ is constant in 
time, so this is quite different from the model studied for example in \cite{CM94,Khoshnevisan}. The difficulty of this problem is twofold. 
First, the product $u\cdot\xi$ is not classically well-defined since the sum of the H\"older regularities of $u$ and $\xi$ 
is slightly below $0$. Second, our space variable $x$ lies in the unbounded space $\R^2$ so that one needs to incorporate 
weights in the H\"older spaces at stake; this causes some difficulty in obtaining the fixed point argument, since one 
would a priori require a larger weight for $u\cdot\xi$ than for $u$ itself.

The first issue is handled thanks to a renormalisation procedure which, informally, consists in subtracting 
an infinite linear term from the original equation. The main trick that spares us from using elaborate 
renormalisation theories is to introduce the ``stationary'' solution $Y$ of the (additive) stochastic heat equation and 
to solve the PDE associated to $v=u e^Y$ instead of $u$. This is analogous to what was done for example 
in \cite{MR1941997,HaPaPi2013}.
The second issue is dealt with by choosing an 
appropriate time-increasing weight for the solution $u$. Roughly speaking, if $\xi$ is weighted by the 
polynomial function $\pa(x)=(1+|x|)^a$ with $a$ small, and $u_s$ is weighted by the exponential function 
$\e_s(x)=e^{s(1+|x|)}$, then $\int_0^t P_{t-s}*(u_s\cdot\xi)(x)\,ds$ requires a weight of order 
$\int_0^t \pa(x)\e_s(x)\, ds$, which is smaller than $\e_t(x)$. This argument already appears in \cite{HaPaPi2013}, 
and probably also elsewhere in the PDE literature.

The solution to the (generalised) parabolic Anderson 
model has already been constructed independently by Gubinelli, 
Imkeller and Perkowski~\cite{GubImkPer} and by Hairer \cite{Hairer2014} 
in dimension $2$ and, to some extent, by 
Hairer and Pardoux~\cite{Etienne} in dimension $3$. (The latter actually 
considers the case of dimension $1$ 
with space-time white noise, but the case of dimension $3$ with spatial 
noise has exactly the same scaling
behaviour, so the proof given there carries through \textit{mutatis mutandis}. The main 
difference is that some of the renormalisation constants
that converge to finite limits in \cite{Etienne} may diverge logarithmically.) However, in all of these results 
the space variable is restricted to a torus, which is the constraint that we lift in this note. 
The construction that we propose here is very specific 
to (\ref{e:PAM}) in dimension $2$: in particular, as it stands, it unfortunately applies 
neither to the generalised parabolic Anderson
model considered in \cite{GubImkPer,Hairer2014}, nor to the case of dimension $3$. Let us also mention the work of Hu~\cite{Hu} who considers a different equation: the usual product $u\cdot\xi$ in (\ref{e:PAM}) is replaced by the Wick product.

Let us now present the main steps of our construction. First, we introduce a mollified noise $\xi_\varepsilon := \rho_\varepsilon * \xi$, where $\rho$ is a compactly supported, even, smooth function on $\R^2$ that integrates to $1$, and $\rho_\varepsilon(x):= \varepsilon^{-2}\rho(\frac{x}{\varepsilon})$ for all $x\in\R^2$. In order to quantify the H\"older regularity of $\xi, \xi_\varepsilon$, we introduce weighted H\"older spaces of distributions, see Section \ref{SectionWeights} below for the general definitions. 
Informally speaking, given a weight $w$ and an exponent $\alpha$, $\CC_w^\alpha$ consists of those
elements of $\CC^\alpha$ that grow at most as fast as $w$ at infinity.
We have the following very simple convergence result, the proof of which is 
given on Page~\pageref{proof:xi} below.

\begin{lemma}\label{Lemma:xi}
For any given $a > 0$, let $\pa(x)=(1+|x|)^a$ on $\R^2$ as above. For every $\varepsilon, \kappa > 0$, $\xi_\varepsilon$ belongs almost surely to $\cC^{-1-\kappa}_{\pa}(\R^2)$. As $\varepsilon\downarrow 0$, $\xi_\varepsilon$ converges in probability to $\xi$ in $\cC^{-1-\kappa}_{\pa}$.
\end{lemma}
From now on, $a$ is taken arbitrarily small. Since, for any fixed $\eps > 0$,
the mollified noise $\xi_\eps$ is actually a smooth function belonging to $\cC^{\alpha}_{\pa}$ for any $\alpha > 0$, the SPDE
\begin{equ}[e:PAMeps]\tag{PAM$_\varepsilon$}
\partial_t u_\varepsilon = \Delta u_\varepsilon + u_\varepsilon\big(\xi_\varepsilon-C_\varepsilon) \;,\qquad 
u_\varepsilon(0,x) = u_0(x)\;,
\end{equ}
 is well-posed, as can be seen for example by using its Feynman-Kac representation.
The constant $C_\varepsilon$ appearing in this equation is required in order to control the limit
$\varepsilon \to 0$ and will be determined later on.

Second, let $G$ be a compactly supported, even, smooth function on $\R^2\backslash\{0\}$, such that $G(x) = \frac{\log |x|}{2\pi}$ whenever $|x|\leq \frac{1}{2}$. Then, there exists a compactly supported smooth function $F$ on $\R^2$ that vanishes on the ball of radius $\frac{1}{2}$ and such that, in the distributional sense, we have:
\begin{equs}\label{Eq:G}
\Delta G(x) = \delta_0(x) + F(x) \;.
\end{equs}
With these notations at hand, we introduce the process
$Y_\varepsilon(x) := G * \xi_\varepsilon(x)$.
By construction, $Y_\varepsilon$ is a smooth stationary process on $\R^2$ that coincides with the solution of the Poisson equation driven by $\xi_\varepsilon$, up to some smooth term:
\[ \Delta Y_\varepsilon(x) = \xi_\varepsilon(x) + F * \xi_\varepsilon(x) \;.\]
From now on, $D_{x_i}$ denotes the differentiation operator with respect to the variable $x_i$, with $i\in\{1,2\}$. More generally, for every $\ell\in\N^2$, we define $D^\ell_x f$ as the map obtained from $f$ by differentiating $\ell_1$ times in direction $x_1$ and $\ell_2$ times in direction $x_2$. We also use the notation $\nabla f =(D_{x_1} f, D_{x_2} f)$.
The following result is a consequence of Lemma \ref{Lemma:xi} together with the smoothing effect of the convolution with $G$ and $D_{x_i}G$.

\begin{corollary}\label{Cor:Y}
For any given $\kappa\in(0,1/2)$, the sequence of processes $Y_\varepsilon$ (resp. $D_{x_i} Y_\varepsilon$) converges in probability as $\varepsilon\rightarrow 0$ in the space $\cC^{1-\kappa}_{\pa}(\R^2)$ (resp. $\cC^{-\kappa}_{\pa}(\R^2)$) towards the process $Y$ (resp. $D_{x_i} Y$) defined by
\[ Y := G * \xi \;,\qquad D_{x_i} Y := D_{x_i} G * \xi \;. \]
\end{corollary}
We introduce $v_\varepsilon(t,x) := u_\varepsilon(t,x) e^{Y_\varepsilon(x)}$ for all $x\in\R^2$ and $t\geq 0$, and we observe that
\begin{equs}
\partial_t v_\varepsilon = \Delta v_\varepsilon + v_\varepsilon (Z_\varepsilon - F*\xi_\varepsilon) - 2\nabla v_\varepsilon \cdot \nabla Y_\varepsilon  \;,\qquad v_\varepsilon(0,x) = u_0(x) e^{Y_\varepsilon(x)} \;,
\end{equs}
where we have introduced the renormalised process
\begin{equs}
Z_\varepsilon(x) := |\nabla Y_\varepsilon(x)|^2-C_\varepsilon \;.
\end{equs}
At this stage we fix the renormalisation constant $C_\varepsilon$ to be given by
\begin{equs}\label{Eq:Constant}
C_\varepsilon := \bbE\left[|\nabla Y_\varepsilon|^2\right] = -\frac{1}{2\pi}\log \varepsilon + \cO(1) \;,
\end{equs}
where the part denoted by $\cO(1)$ converges to a constant (depending on the choice of $G$ and $\rho$) as $\varepsilon \to 0$, we refer to the end of Section \ref{SectionBounds} for the calculation. The following result, which is proven on Page~\pageref{proof:Z}, shows that this sequence of renormalised processes also converges in an appropriate space. We refer to Nualart~\cite{Nualart2006} for details on Wiener chaoses.
\begin{proposition}\label{Prop:Z}
For any given $\kappa\in(0,1/2)$, the collection of processes $Z_\varepsilon$ converges in probability as $\varepsilon\rightarrow 0$, in the space $\cC^{-\kappa}_{\pa}(\R^2)$, towards the generalised process $Z$ defined as follows: for every test function $\eta$, $\langle Z,\eta \rangle$ is the random variable in the  second homogeneous 
Wiener chaos associated to $\xi$ represented by the $L^2(dz\,d\tilde{z})$ function
\[(z,\tilde{z})\mapsto \int \sum_{i=1,2} D_{x_i} G(z-x) D_{x_i} G(\tilde{z}-x) \eta(x) dx \;.\]
\end{proposition}
We are now able to set up a fixed point argument for the process $v_\varepsilon$ with controls that are uniform 
in $\varepsilon$. The precise statement of the main result of this article requires some notation: in this introduction, we provide a weaker but more readable version of the statement and we refer to Section \ref{sectionPicard} for the details.
\begin{theorem}\label{ThMain}
Let $u_0$ be a H\"older distribution with regularity better than $-1$, and that grows at most exponentially fast at infinity. The sequence of processes $v_\varepsilon$ converges uniformly on all compact sets of $(0,\infty)\times\R^2$, in probability as $\varepsilon\rightarrow 0$, to a limit $v$ which is the unique solution of
\begin{equs}
\partial_t v = \Delta v + v(Z- F*\xi) - 2 \nabla v \cdot \nabla Y \;,\qquad v(0,x) = u_0(x) e^{Y(x)}\;.
\end{equs}
As a consequence, $u_\varepsilon$ converges in probability towards the process $u=ve^{-Y}$.
\end{theorem}

\section{Weighted H\"older spaces}\label{SectionWeights}
In this section, we introduce the appropriate weighted spaces that will allow us to set up a fixed point argument associated to (\ref{e:PAM}). We work in $\R^d$ for a general dimension $d\in\N$, even though we will apply these results to $d=2$ in the next sections.
\begin{definition}
A function $w:\R^d\rightarrow(0,\infty)$ is a weight if there exists a positive constant $C>0$ such that
\[ C^{-1} \leq \sup_{|x-y| \leq 1} \frac{w(x)}{w(y)} \leq C \;.\]
\end{definition}
In this article, we will consider two families of weights indexed by $a,\ell\in\R$:
\[ \pa(x) := (1+|x|)^a \;,\qquad \e_\ell(x) := \exp\big(\ell (1+|x|)\big) \;. \]
Observe that the constant $C$ can be taken uniformly for all $\pa$ and $\e_\ell$, as long as $a$ and $\ell$ lie in a compact domain of $\R^2$.
We can now consider weighted versions of the usual spaces of H\"older functions $\cC^\alpha(\R^d)$.
\begin{definition}
For $\alpha\in(0,1)$, $\cC^\alpha_w(\R^d)$ is the space of functions $f:\R^d\rightarrow\R$ such that
\[ \| f \|_{\alpha,w} := \sup_{x\in\R^d} \frac{|f(x)|}{w(x)} + \sup_{|x-y|\leq 1}\frac{|f(x)-f(y)|}{w(x) |x-y|^\alpha} < \infty \;.\]
More generally, for every $\alpha > 1$, we define $\cC^{\alpha}_w(\R^d)$ recursively
as the space of functions $f$ which admit first order derivatives and such that
\[ \| f \|_{\alpha,w} := \sup_{x\in\R^d} \frac{|f(x)|}{w(x)} + \sum_{i=1}^d\| D_{x_i} f \|_{\alpha-1,w} < \infty \;.\]
\end{definition}
We then extend this definition to negative $\alpha$. To this end, we define for every $r\in\N$, the space $\cB^r_1$ of all smooth functions $\eta$ on $\R^d$, which are compactly supported in the unit ball of $\R^d$ and whose $\cC^r$ norm is smaller than $1$. We will use the notation $\eta^\lambda_x$ to denote the function $y\mapsto \lambda^{-d} \eta\big(\frac{y-x}{\lambda}\big)$.
\begin{definition}\label{Def:HolderWeight}
For every $\alpha < 0$, we set $r:=-\lfloor \alpha \rfloor$ and we define $\cC^\alpha_w(\R^d)$ as the space of distributions $f$ on $\R^d$ such that
\[ \| f \|_{\alpha,w} := \sup_{x\in\R^d} \sup_{\eta\in\cB^r_1} \sup_{\lambda\in (0,1]} \frac{|f(\eta_x^\lambda)|}{w(x)\lambda^\alpha} < \infty \;.\]
\end{definition}
In order to deal with the regularity of random processes, it is convenient to have a characterisation of $\cC^{\alpha}_w$ that only relies on a countable number of test functions. To state such a characterisation, we need some notation. For any $\psi\in\cC^r$, we set
\begin{equs}
\psi^n_x(y):= 2^{\frac{nd}{2}}\psi((y_1-x_1)2^n,\ldots,(y_d-x_d)2^n)\;,\quad x,y\in\R^d\;,\quad n\geq 0\;.
\end{equs}
We also define $\Lambda_n:=\{(2^{-n}k_i)_{i=1\ldots d}:k_i\in\Z\}$.
\begin{proposition}\label{Prop:CharactDistrib}
Let $\alpha<0$ and $r > |\alpha|$. There exists a finite set $\Psi$ of compactly supported functions in $\cC^r$, as well as a compactly supported function $\varphi\in\cC^r$ such that $\{\varphi^0_x, x\in\Lambda_0\} \cup\{ \psi^n_{x}, n\geq 0, x\in \Lambda_n,\psi\in\Psi\}$ forms an orthonormal basis of $\R^d$, and such that for any distribution $\xi$ on $\R^d$, the following equivalence holds: $\xi \in \cC^{\alpha}_{w}$ if and only if $\xi$ belongs to the dual of $\cC^{r}$ and
\begin{equs}\label{Eq:NormCaw}
\sup_{n\geq 0} \sup_{\psi\in\Psi} \sup_{x\in\Lambda_n}\frac{|\left\langle\xi,\psi_x^{n}\right\rangle|}{w(x)2^{-{nd\over 2}-n\alpha}} + \sup_{x\in\Lambda_0}\frac{|\left\langle\xi, \varphi_x^{0}\right\rangle|}{w(x)}< \infty \;. 
\end{equs}
\end{proposition}

\begin{proof}
This result is rather standard and is obtained by a wavelet analysis, see~\cite{Meyer,Ingrid} or~\cite[Prop.~3.20]{Hairer2014}. In these references, the spaces are not weighted, but since all the arguments needed for the proof are local, it suffices to use the fact that $\frac{w(y)}{w(x)}$ is bounded from above and below uniformly over all $x,y$ such that $|x-y| \leq 1$ to obtain our statement.
\end{proof}

\begin{remark}
If $\xi$ is a linear transformation acting on the linear span of the functions $\phi^0_x$, $\psi_x^n$ such that (\ref{Eq:NormCaw}) is finite, then $\xi$ can be extended uniquely to an element of $\cC^{\alpha}_w$.
\end{remark}

We are now in position to characterise the regularity of the noise.
\begin{proof}[of Lemma~\ref{Lemma:xi}]\label{proof:xi} We work in dimension $d=2$. Set $\alpha = -1-\kappa$ with $\kappa > 0$. By Proposition~\ref{Prop:CharactDistrib}, it suffices to show that almost surely
\begin{equation*}
\sup_{n\geq 0}\sup_{\psi\in\Psi}\sup_{x\in\Lambda_n}\frac{|\left\langle\xi,\psi_x^{n}\right\rangle|}{2^{-n(1+\alpha)}\pa(x)} \lesssim 1  \;,\qquad \sup_{x\in\Lambda_0}\frac{|\left\langle\xi, \varphi_x^{0}\right\rangle|}{\pa(x)} \lesssim 1\;.
\end{equation*}
We restrict to the first bound, since the second is simpler. For any integer $p\geq 1$, we write
\begin{equs}
\bbE\left[ \sup_{n\geq 0}\sup_{\psi \in \Psi}\sup_{x\in\Lambda_n}\Big( \frac{|\left\langle\xi,\psi_x^{n}\right\rangle|}{2^{-n(\alpha+1)}\pa(x)} \Big)^{2p}\right] &\lesssim \sum_{n\geq 0}\sum_{\psi \in \Psi}\sum_{x\in\Lambda_n}\frac{2^{2np(\alpha+1)}}{\pa(x)^{2p}}\big(\bbE \left\langle\xi,\psi_x^n\right\rangle^2\big)^p\\
&\lesssim \sum_{n\geq 0}\sum_{\psi \in \Psi}\sum_{x\in\Z^2}\frac{2^{2np(\alpha+1)}}{\pa(x)^{2p}}2^{2n}\;.
\end{equs}
At the first line, we used the equivalence of moments of Gaussian random variables. At the second line, we used the following facts: the restriction of $\Lambda_n$ to the unit ball of $\R^2$ has at most of the order of $2^{2n}$ elements, the $L^2$ norm of $\psi_x^{n}$ is $1$ and $\pa$ is a weight. Recall that $\alpha < -1$, $\Psi$ is a finite set and $\pa(x)=(1+|x|)^a$. Taking $p$ large enough, we deduce that the triple sum converges, so that $\xi$ admits a modification that almost surely belongs to $\cC^{\alpha}_{\pa}$. We now turn to $\|\xi_\varepsilon - \xi\|_{\alpha,\pa}$: the computation is very similar, the only difference rests on the term
\begin{equs}\label{Eq:RegulXi}
\bbE \left\langle\xi-\xi_\varepsilon,\psi_x^{n}\right\rangle^2 &= \|\psi_0^n-\rho_\varepsilon*\psi_0^n\|^2_{L^2} \lesssim 1\wedge (\varepsilon^2 2^{2n})\;.
\end{equs}
Let $n_0$ be the smallest integer such that $2^{-n_0} \leq \varepsilon$. For $p$ large enough, we obtain
\begin{equs}
\bbE \left[\sup_{n\geq 0}\sup_{\psi \in \Psi}\sup_{x\in\Lambda_n}\Big( \frac{|\left\langle\xi-\xi_\varepsilon,\psi_x^{n}\right\rangle|}{2^{-n(\alpha+1)}\pa(x)} \Big)^{2p}\right] &\lesssim \sum_{x\in\Z^2}\sum_{n\geq 0}\frac{2^{2n+2np(\alpha+1)}}{\pa(x)^{2p}} (1 \wedge \varepsilon^{2p} 2^{2np}) \\
&\lesssim \sum_{n<n_0} \varepsilon^{2p} 2^{2n \big( p(\alpha+2) + 1\big)} + \sum_{n\geq n_0} 2^{2n \big( p(\alpha+1) + 1\big)}\;.
\end{equs}
Since $\alpha=-1-\kappa < -1$, taking $p$ large enough, we get that the second sum on the r.h.s.~is bounded by a term of order $\varepsilon^{-2\big(1+p(\alpha+1)\big)}$. Then, according as $p(\alpha+2)+1$ is negative, null or positive, the first sum on the r.h.s.~is bounded by a term of order $\varepsilon^{2p}$, $\varepsilon^{2p}|\log_2\varepsilon|$ or $\varepsilon^{-2\big(1+p(\alpha+1)\big)}$. Consequently, for $p$ large enough $\bbE \|\xi_\varepsilon - \xi\|_{\alpha,\pa}^{2p} \rightarrow 0$ as $\varepsilon\downarrow 0$.
\end{proof}

Let $w_f$ and $w_g$ be two weights on $\R^d$. We have the following elementary extension of the classical 
theorem \cite[Thm~2.52]{BookChemin}.
\begin{theorem}\label{Th:Young}
Let $f\in\cC^\alpha_{w_f}$ and $g\in\cC^\beta_{w_g}$ where $\alpha < 0$ and $\beta > 0$ with $\alpha+\beta > 0$. Then there exists a continuous bilinear map $(f,g)\mapsto f\cdot g$ from $\cC^\alpha_{w_f}\times\cC^\beta_{w_g}$ into $\cC^\alpha_{w_f w_g}$ that extends the classical multiplication of smooth functions.
\end{theorem}
\begin{remark}
The space $\cC^\alpha$ defined in Section 2 coincides with the usual Besov space $\cB^{\alpha}_{\infty,\infty}$. Indeed, they enjoy the same characterisation in a wavelet analysis, see \cite[Prop 3.20]{Hairer2014} and \cite[Section 6.10]{Meyer}.
\end{remark}
\begin{proof}
Let $\chi$ be a compactly supported, smooth function on $\R^d$ such that $\sum_{k\in\Z^d} \chi(x-k) = 1$ for all $x\in\R^d$. For simplicity, we set $\chi_k(\cdot):=\chi(\cdot-k)$. Writing $\|\cdot\|_\alpha$ for the 
$\alpha$-H\"older norm without weight (i.e.\ with weight $1$), observe that $h \in \CC^\alpha_w$ if and only
if $\|h\chi_k\|_\alpha \lesssim w(k)$
hold uniformly over all $k\in\Z^d$, and $\|h\|_{\alpha,w}$ is equivalent to the smallest possible bound. From \cite[Thm~2.52]{BookChemin}, we know that $f\chi_k \cdot g\chi_\ell$ is well-defined for all $k,\ell\in\Z^d$, and that the bound $\|f\chi_k\cdot g\chi_\ell \|_{\alpha} \lesssim \| f\chi_k \|_{\alpha}\| g\chi_\ell \|_{\beta}$ holds. Consequently, we get
\begin{equs}
\| f\chi_k\cdot g \chi_\ell \|_{\alpha} \lesssim w_f(k) w_g(\ell)\| f \|_{\alpha,w_f}\| g\|_{\beta,w_g}\;,
\end{equs}
uniformly over all $k,\ell\in\Z^d$. Since the number of non-zero terms among $\{\langle f\chi_k\cdot g \chi_\ell, \eta_x \rangle, k,\ell\in\Z^d\}$ is uniformly bounded over all $\eta\in\cB^r_1$, all $x\in\R^d$ and all $f,g$ as in the statement, we deduce that $f\cdot g := \sum_{k,\ell\in\Z^d} f\chi_k\cdot g\chi_\ell$ is well-defined and that $\| f\cdot g \|_{\alpha,w_f w_g} \lesssim \|f\|_{\alpha,w_f} \|g\|_{\beta,w_g}$ holds. Finally, the multiplication of \cite[Thm~2.52]{BookChemin} extends the classical multiplication of smooth functions, therefore, from our construction, it is plain that this property still holds in our case.
\end{proof}

Let now $P_t(x) := (4\pi t)^{-\frac{d}{2}}e^{-|x|^2 / 4t}$ be the heat kernel in dimension $d$. We write $P_t * f$ for the spatial convolution of $P_t$ with a function/distribution $f$ on $\R^d$. We have the following regularisation property which is a slight variant of well-known facts. 
\begin{lemma}\label{Lemma:HeatKernel}
For every $\beta \geq \alpha$ and every $f\in\cC^\alpha_{\e_\ell}$, we have
\[ \| P_t f \|_{\beta, \e_\ell} \lesssim t^{-\frac{\beta-\alpha}{2}} \|f\|_{\alpha,\e_\ell}  \;,\]
uniformly over all $\ell$ in a compact set of $\R$ and all $t$ in a compact set of $[0,\infty)$.
\end{lemma}
\begin{proof}
We use a decomposition of the heat kernel $P_t(x)=P_+(t,x) + P_-(t,x)$ where $P_-$ is smooth and $P_+$ is supported in the unit ball centred at $0$, we refer the reader to Lemma 5.5 in~\cite{Hairer2014} for instance. Using the decay properties of the heat kernel, the statement regarding $P_-$ is easy to check. Concerning the singular part, one writes $P_+ = \sum_{n\geq 0} P_n$ where each $P_n$ is a smooth function supported in the 
parabolic annulus $\{(t,x): 2^{-n-1} \leq |t|^\frac{1}{2} + |x| \leq 2^{-n+1}\}$ and such that $P_n(t,x)=2^{dn}P_0(2^{2n} t, 2^n x)$. Then, we get
\[ |\langle f , \eta^\lambda_x(\cdot-y)\rangle| \lesssim \lambda^\alpha \e_\ell(x+y) \;,\qquad |\langle f , D_x^k P_n(t,\cdot-y)\rangle| \lesssim 2^{-n(\alpha-|k|)} \e_\ell(y) \;,\]
uniformly over all $\eta\in\cB^r_1$, all $x,y\in\R^d$, all $t>0$, all $n\geq 0$ and all $k\in\N^2$. Notice that $P_n(t,\cdot)$ vanishes as soon as $n\geq 1-\frac{1}{2}\log_2 t$. Consequently,
\begin{equs}
|\langle P_+(t)*f , \eta_x^\lambda\rangle| \lesssim \e_\ell(x)(\lambda^\alpha \wedge t^{\frac{\alpha}{2}}) \;,\quad |\langle f ,D_x^k P_+(t,\cdot-x)\rangle| \lesssim \e_\ell(x) t^{\frac{\alpha-|k|}{2}}\;,
\end{equs}
so that the statement follows by interpolation.
\end{proof}

\section{Bounds on $Y$ and $Z$}\label{SectionBounds}
Let us collect a few facts on the behaviour of smooth functions with a singularity at the origin; we refer to \cite[Sec.~10.3]{Hairer2014} for proofs. For any smooth function $K:\R^d\backslash\{0\}\rightarrow\R$ and any real number $\zeta$, we define
\begin{equs}
\$ K \$_{\zeta;m} = \sup_{|k|\leq m} \sup_{x\in\R^d} \|x\|^{|k|-\zeta}|D^k_x K(x)|\;,
\end{equs}
where the first supremum runs over $k\in\N^d$ and $|k|=\sum_i k_i$. We say that $K$ is of order $\zeta$ if $\$ K \$_{\zeta;m} < \infty$ for all $m\in\N$.
Recall $\rho_\varepsilon$ from the introduction, and define $K_\varepsilon = K*\rho_\varepsilon$. If $K$ is of order $\zeta\in(-d,0)$ then for all $m\in\N$, there exists $C>0$ such that
$\$ K_\varepsilon \$_{\zeta;m} \leq C \$ K \$_{\zeta;m}$, uniformly over $\varepsilon \in (0,1]$.
Furthermore, for all $\bar\zeta\in[\zeta-1,\zeta)$, there exists a constant $C>0$ such that
\begin{equs}
\$ K-K_\varepsilon \$_{\bar\zeta;m} \leq C \varepsilon^{\zeta-\bar\zeta}\$ K\$_{\zeta;m+1}\;.
\end{equs}
If $K_1$ and $K_2$ are of order $\zeta_1$ and $\zeta_2$ respectively, then $K_1K_2$ is of order $\zeta=\zeta_1+\zeta_2$ and we have the bound
\begin{equs}
\$ K_1 K_2 \$_{\zeta;m} \leq C \$ K_1 \$_{\zeta_1;m}\$ K_2 \$_{\zeta_2;m}\;,
\end{equs}
where $C$ is a positive constant.

Assume that $\zeta_1\wedge\zeta_2 > -d$. We set $\zeta=\zeta_1+\zeta_2+d$. If $\zeta<0$, then $K_1*K_2$ is of order $\zeta$ and we have the bound
\begin{equs}\label{Eq:BoundsConvol}
\$ K_1 * K_2 \$_{\zeta;m} \leq C \$ K_1 \$_{\zeta_1;m}\$ K_2 \$_{\zeta_2;m}\;.
\end{equs}
On the other hand, if $\zeta\in\R_+\backslash\N$ and $K_1,K_2$ are compactly supported, then the function
\begin{equs}
K(x) = (K_1*K_2)(x) - \sum_{|k| < \zeta}\frac{x^k}{k!}D^k_x(K_1*K_2)(0) \;,
\end{equs}
is of order $\zeta$ and a bound similar to (\ref{Eq:BoundsConvol}) holds, but with the constant $C$ depending 
on the size of the supports in general.

\smallskip

We will apply these bounds to the function $G$ defined in the introduction. Since $G$ is smooth on $\R^2\backslash\{0\}$, compactly supported and satisfies $G(x) = \frac{\log |x|}{2\pi}$ in a neighbourhood of the origin, it is a function with a singularity of order $\zeta$, for all $\zeta < 0$, according to our definition. From now on, we set $\rho^{*2}=\rho*\rho$ and we assume without loss of generality that $\rho$, $\rho^{*2}$ are supported in the unit ball of $\R^2$.

\begin{lemma}\label{Lemma:Z}
Fix $\kappa \in (0,1)$. We have the bounds
\begin{equ}
\bbE\left[ |Z(\eta_x^\lambda)|^2 \right] \lesssim \lambda^{-\kappa} \;,\;\;
\bbE\left[ |Z_\varepsilon(\eta_x^\lambda)|^2 \right] \lesssim \lambda^{-\kappa}\;,\;\;
\bbE\left[ |Z_\varepsilon(\eta_x^\lambda)-Z(\eta_x^\lambda)|^2 \right] \lesssim \lambda^{-5\kappa} \varepsilon^{\kappa} \;,
\end{equ}
uniformly over all $\varepsilon,\lambda\in(0,1)$, all $x \in \R^2$ and all $\eta\in\cB_1^r$.
\end{lemma}
\begin{proof}
By translation invariance, it suffices to consider $x=0$. The random variables $Z(\eta^\lambda)$, $Z_\varepsilon(\eta^\lambda)$ and $Z_\varepsilon(\eta^\lambda)-Z(\eta^\lambda)$ all belong to the second \textit{homogeneous} Wiener chaos associated with the noise $\xi$. This is because the constant $C_\varepsilon$ has been chosen to cancel the $0$-th Wiener chaos component of $|\nabla Y_\varepsilon|^2$. We start with the second bound of the statement:
\begin{equs}
\bbE\left[ |Z_\varepsilon(\eta^{\lambda})|^2 \right]&= \sum_{i=1}^2 \int_{z,\tilde{z}}\Big(\int \eta^\lambda(x)D_{x_i} G_\varepsilon(z-x)D_{x_i} G_\varepsilon(\tilde{z}-x)dx\Big)^2dz\,d\tilde{z}\\
&=\sum_{i=1}^2 \int\!\!\int \eta^\lambda(x)\eta^\lambda(x')\Big((D_{x_i} G_\varepsilon)*(D_{x_i} G_\varepsilon)(x-x')\Big)^2 dx\,dx'\;,
\end{equs}
so that the bounds at the beginning of the section yield the desired result. The first bound of the statement follows by replacing $G_\varepsilon$ by $G$ in the expression above. We turn to the proof of the third bound. To that end, we write
\begin{equs}
\bbE\left[ |Z_\varepsilon(\eta^\lambda)-Z(\eta^\lambda)|^2 \right]=\sum_{i=1}^2 \int\!\!\int \eta^\lambda(x)\eta^\lambda(x') H_{\varepsilon,i}(x-x') dx\,dx'\;,
\end{equs}
where
\begin{equs}
H_{\varepsilon,i}(y) &= \Big(\bigl(D_{x_i} (G_\varepsilon-G)\bigr)*D_{x_i}G_\varepsilon\Big)\cdot\Big(\bigl(D_{x_i} (G_\varepsilon+G)\bigr)*D_{x_i}G_\varepsilon\Big)(y)\\
&\quad -\Big(\bigl(D_{x_i} (G_\varepsilon-G)\bigr)*D_{x_i}G\Big)\cdot\Big(\bigl(D_{x_i} (G_\varepsilon+G)\bigr)*D_{x_i}G\Big)(y)\;,
\end{equs}
so that, once again, the bounds on the behaviour of singular functions at the origin yield the asserted bound.
\end{proof}

\begin{proof}[of Proposition~\ref{Prop:Z}]\label{proof:Z}
Let $L$ denote an arbitrary element among $Z$, $Z_\varepsilon$ and $Z-Z_\varepsilon$. Using the equivalence of moments of elements in inhomogeneous Wiener chaoses of finite order, we obtain
\begin{equs}
\bbE\left[ \sup_{n\geq 0} \sup_{x\in\Lambda_n} \Big(\frac{L(\psi_x^n)}{\pa(x)2^{-n\alpha -n}}\Big)^{2p}\right] \lesssim \sum_{k\in\Z^2}\frac{1}{\pa(k)^{2p}}\sum_{n\geq 0}\sum_{x\in\Lambda_n \cap B(k,1)} \frac{\bbE[L(\psi_x^n)^2]^p}{2^{-n\alpha 2p -2np} } \;.
\end{equs}
When $L$ is equal to $Z$ or $Z_\varepsilon$, Lemma \ref{Lemma:Z} ensures that $\bbE[L(\psi_x^n)^2] \lesssim 2^{-2n+\kappa n}$ uniformly over all $x$, $n$, and $\varepsilon$. Moreover, $\#( \Lambda_n\cap B(k,1)) \lesssim 2^{2n}$, so that
\[ \bbE\left[ \sup_{n\geq 0} \sup_{x\in\Lambda_n} \Big(\frac{L(\psi_x^n)}{\pa(x)2^{-n\alpha -n}}\Big)^{2p}\right] \lesssim\sum_{k\in\Z^2}\frac{1}{\pa(k)^{2p}} \sum_{n\geq 0} 2^{np(2\alpha+\kappa)+2n}\;.\]
This quantity is finite for $\alpha=-\kappa$ and $p$ large enough. Therefore, $Z$ and $Z_\varepsilon$ belong to $\cC^{-\kappa}_{\pa}$.\\
Regarding $Z-Z_\varepsilon$, Lemma \ref{Lemma:Z} ensures that $\bbE[(Z-Z_\varepsilon)(\psi_x^n)^2] \lesssim \varepsilon^{\kappa} 2^{-2n +5\kappa n}$ uniformly over all $x$, $n$ and $\varepsilon$. Then, the same arguments as before yield
\[ \bbE\left[ \sup_{n\geq 0} \sup_{x\in\Lambda_n} \Big(\frac{(Z-Z_\varepsilon)(\psi_x^n)}{\pa(x)2^{-n\alpha -n}}\Big)^{2p}\right] \lesssim \sum_{k\in\Z^2}\frac{1}{\pa(k)^{2p}} \sum_{n\geq 0} \varepsilon^{\kappa p} 2^{np(2\alpha + 5\kappa) + 2n}\;,\]
so that, choosing for instance $\alpha=-3\kappa$ and $p$ large enough, one gets the bound $\bbE\left[\| Z-Z_\varepsilon\|_{-3\kappa,\pa}\right] \lesssim \varepsilon^{\frac{\kappa}{2}}$ uniformly over all $\varepsilon\in (0,1]$, thus concluding the proof.
\end{proof}

\begin{proof}[of Corollary \ref{Cor:Y}]
Since $G$ is compactly supported and coincides with the Green function of the Laplacian in a neighbourhood of the origin, the classical Schauder estimates~\cite{Simon} imply that for any $\alpha\in\R$, the bounds
\begin{equs}
\|G*f\|_{\alpha+2} \lesssim \|f\|_\alpha \;,\quad\|D_{x_i}G*f\|_{\alpha+1} \lesssim \|f\|_\alpha\;,
\end{equs}
hold uniformly over all $f \in\cC^\alpha$. Recall the functions $\chi_k, k\in\Z^d$ from the proof of Theorem \ref{Th:Young}. Since $G$ is compactly supported, we deduce from the bounds above that
\begin{equs}
\|G*(f\chi_k)\|_{\alpha+2} \lesssim w(k)\|f\|_{\alpha,w} \;,\quad\|D_{x_i}G*(f\chi_k)\|_{\alpha+1} \lesssim w(k)\|f\|_{\alpha,w}\;,
\end{equs}
uniformly over all $k\in\Z^d$ and all $f\in\cC^\alpha_w$. For fixed $x$, only a bounded number of $\{\chi_k(x),k\in\Z^d\}$ are non-zero, uniformly over all $x\in\R^d$. Since $f=\sum_{k\in\Z^d} f\chi_k$, we deduce that
\begin{equs}
\|G*f\|_{\alpha+2,w} \lesssim \|f\|_{\alpha,w} \;,\quad\|D_{x_i}G*f\|_{\alpha+1,w} \lesssim \|f\|_{\alpha,w}\;,
\end{equs}
uniformly over all $f\in\cC^\alpha_w$. This being given, the statement is a direct consequence of Lemma \ref{Lemma:xi}. 
\end{proof}

We conclude this section with the computation of the renormalisation constant $C_\varepsilon$. Recall that $\rho$, $\rho^{*2}$ and $G$ are compactly supported. We let $G_\varepsilon$ be the compactly supported, smooth function $G * \rho_\varepsilon$. We have
\begin{equs}
C_\varepsilon = \bbE\left[|\nabla Y_\varepsilon|^2\right] = \sum_{i=1,2}\int_{x\in\R^2} D_{x_i} G_\varepsilon(x) D_{x_i} G_\varepsilon(x) dx= -\int_{x\in\R^2} G_\varepsilon(x) \Delta G_\varepsilon(x) dx \;,
\end{equs}
where we used a simple integration by parts to get the last identity. By (\ref{Eq:G}), we have $\Delta G_\varepsilon = \rho_\varepsilon + F_\varepsilon$, where $F_\varepsilon = F*\rho_\varepsilon$. The latter is a compactly supported, smooth function that vanishes on the centred ball of radius $1/2 - \varepsilon$. Hence, uniformly over all $\varepsilon\in (0,1]$, the function $G_\varepsilon F_\varepsilon$ is smooth and compactly supported so that its integral is uniformly bounded. On the other hand, since $\rho$ is even, $\rho^{*2}$ integrates to $1$ and $G(x) = \frac{1}{2\pi} \log |x|$ for all $x\in B(0,1/2)$, we get 
\begin{equs}
-\int G_\varepsilon(x)\rho_\varepsilon(x) dx = -\int G(x)\rho^{*2}_\varepsilon(x) dx= \frac{1}{2\pi} \log \varepsilon^{-1} - \frac{1}{2\pi} \int \log |x| \rho^{*2}(x) dx\;.
\end{equs}
The first term on the right gives the diverging term of the renormalisation constant, while the second term is finite. This concludes the computation.

\section{Picard iteration}\label{sectionPicard}
For any $r>0$, $\ell \in \R$ and $T > 0$, we consider the Banach space $\cE_{\ell,T}^{r}$ of all continuous functions $v$ on $(0,T]\times\R^2$ such that
\[ \$ v \$_{\ell,T,r} := \sup_{t\in(0,T]}\frac{\|v(t,\cdot)\|_{r,\e_{\ell+t}}}{t^{-1+\kappa}}  < \infty\;.\]
This being given, we have the following precise statement of our main result.

\begin{theorem}\label{ThMainPrecise}
Let $\ell \in \R$ and $T > 0$. Consider an initial condition $u_0\in \cC^{-1+4\kappa}_{\e_{\ell}}$. For all $\ell'>\ell$, the sequence of processes $v_\varepsilon$ converges in probability as $\varepsilon\rightarrow 0$ in the space $\cE_{\ell',T}^{1+2\kappa}$ to a limit $v$ which is the unique solution of
\begin{equs}
\partial_t v = \Delta v + v(Z- F*\xi) - 2 \nabla v \cdot \nabla Y \;,\qquad v(0,x) = u_0(x) e^{Y(x)}\;.
\end{equs}
As a consequence, $u_\varepsilon$ converges in probability in $\cE_{\ell',T}^{1-\kappa}$ towards the process $u=ve^{-Y}$.
\end{theorem}

The rest of this section is devoted to the proof of this result. Fix $\kappa \in (0,\frac{1}{4})$, and let the parameter $a$ appearing in the weight $\pa$ be any value in $(0,\frac{\kappa}{2})$. Let $g,h^{(1)},h^{(2)}\in\cC^{-\kappa}_{\pa}$ and $f\in\cC^{-1+4\kappa}_{e_{\ell}}$ be given. We define the map $v\mapsto \cM_{T,f} v$ as follows:
\[ \cM_{T,f} v(t) = \int_0^t P_{t-s} * \Big( v_s \cdot g + D_{x_i} v_s \cdot h^{(i)} \Big) ds + P_t * f\;.\]
In this equation, there is an implicit summation over $i\in\{1,2\}$. This convention will be in force for the rest of the article.

\begin{proposition}\label{Prop:FixedPt}
Take $\ell_0 \in \R$. For any given $g,h^{(1)},h^{(2)} \in\cC^{-\kappa}_{\pa}$ and any $f\in\cC^{-1+4\kappa}_{\e_{\ell_0}}$, the map $\cM_{T,f}$ admits a unique fixed point $v \in \cE_{\ell_0,T}^{1+2\kappa}$. Furthermore, the solution map $(g,h^{(1)},h^{(2)},f) \mapsto v$ is continuous.
\end{proposition}
\begin{proof}
The parameter $r$ in the space $\cE_{\ell,T}^{r}$ is taken to be equal to $1+2\kappa$. Since this value is fixed until the end of the proof, we do not write the subscript $r$ in the associated norm.\\
First, Lemma \ref{Lemma:HeatKernel} ensures that $\|P_t*f\|_{1+2\kappa,\e_{\ell+t}} \lesssim t^{-1+\kappa} \|f\|_{-1+4\kappa,\e_\ell}$ uniformly over all $t$ in any given compact interval of $\R_+$.
Second, using Theorem \ref{Th:Young} and the simple inequality
\[ \sup_{x\in\R^2} \frac{\pa(x)\e_{\ell+s}(x)}{\e_{\ell+t}(x)} \leq e^{-a}\Big(\frac{a}{t-s}\Big)^a \;,\]
we obtain
\begin{equs}
\|v_s \cdot g + D_{x_i} v_s \cdot h^{(i)}\|_{-\kappa,\e_{\ell+t}} &\lesssim (t-s)^{-a} \|v_s\|_{1+2\kappa,\e_{\ell+s}}\big(||g||_{-\kappa,\pa}+||h^{(i)}||_{-\kappa,\pa}\big)\\
&\lesssim (t-s)^{-a} s^{-1+\kappa} \$ v\$_{\ell,T}\big(||g||_{-\kappa,\pa}+||h^{(i)}||_{-\kappa,\pa}\big)\;,
\end{equs}
uniformly over all $s,t$ in a compact set of $\R_+$ and all $\ell$ in a compact set of $\R$. Then, by Lemma \ref{Lemma:HeatKernel} and using $a < \kappa/2$, we obtain
\begin{equs}\label{Eq:FixedPt}
\bigg\|\int_0^t P_{t-s} &* \Big( v_s \cdot g + D_{x_i} v_s \cdot h^{(i)} \Big) \,ds\bigg\|_{1+2\kappa,\e_{\ell+t}}\\
 &\lesssim \int_0^t (t-s)^{-\frac{1}{2}-2\kappa} s^{-1+\kappa} ds \,\$ v\$_{\ell,T}\big(||g||_{-\kappa,\pa}+||h^{(i)}||_{-\kappa,\pa}\big)\\
 &\lesssim t^{-1+\kappa} T^{\frac{1}{2}-2\kappa}\$ v\$_{\ell,T}\big(||g||_{-\kappa,\pa}+||h^{(i)}||_{-\kappa,\pa}\big) \;,
\end{equs}
uniformly over all $t\in(0,T]$. This ensures that $\cM_{T,f}(v) \in \cE_{\ell,T}^{1+2\kappa}$. Furthermore we have
\begin{equs}\label{Eq:Contraction}
\$ \cM_{T,f} v-\cM_{T,f} \bar v\$_{\ell,T} \lesssim T^{\frac{1}{2}-2\kappa} \$ v-\bar v \$_{\ell,T} \big(\|g\|_{-\kappa,\pa}+\|h^{(i)}\|_{-\kappa,\pa}\big) \;,
\end{equs}
uniformly over all $\ell$ in a compact set of $\R$, all $T$ in a compact set of $\R_+$, all $f\in\cC^{-1+4\kappa}_{e_\ell}$ and all $v,\bar v\in \cE_{\ell,T}$. (Here and below we write $\cE_{\ell,T}$ instead of $\cE_{\ell,T}^{1+2\kappa}$
for conciseness.)
Consequently, there exists $T^*>0$ such that $\cM_{T^*,f}$ is a contraction on $\cE_{\ell,T^*}$, uniformly over all $\ell \in [\ell_0,\ell_0+T]$ and all $f\in\cC^{-1+4\kappa}_{e_\ell}$. Fix an initial condition $f \in \cC^{-1+4\kappa}_{e_{\ell_0}}$. To obtain a fixed point for the map $\cM_{T,f}$, we proceed by iteration. The map $\cM_{T^*,f}$ admits a unique fixed point $v^* \in \cE_{\ell_0,T^*}$. If $T^* \geq T$, we are done. Otherwise, set $f^* := v^*_{T^*/2} \in \cC^{1+2\kappa}_{e_{\ell_0^*}}$, where $\ell_0^*=\ell_0+T^*/2$. Since $\ell_0^* \leq \ell_0+T$, the map $\cM_{T^*,f^*}$ is again a contraction on $\cE_{\ell_0^*,T^*}$, so that it admits a unique fixed point $v^{**}\in\cE_{\ell_0^*,T^*}$. We define $v_s:=v^*_s$ for all $s\in(0,T^*/2]$ and $v_s:=v_{s-T^*/2}^{**}$ for all $s\in(T^*/2,3T^*/2]$. A simple calculation shows that $v$ is a fixed point of $\cM_{\frac{3T^*}{2},f}$ and that $v\in\cE_{\ell_0,3T^*/2}$. Suppose that $\bar v$ is another fixed point. By the uniqueness of the fixed point on $(0,T^*]$, we deduce that $v^*$ and $\bar v$ coincide on this interval. Moreover, a simple calculation shows that $(\bar v_{s+\frac{T^*}{2}},s\in(0,T^*])$ is necessarily a fixed point of $\cM_{T^*,f^*}$ so that it coincides with $v^{**}$. Iterating this argument ensures existence and uniqueness of the fixed point on any interval $[0,T]$.

We turn to the continuity of the solution map with respect to $f$, $g$ and $h^{(i)}$. Let $\bar{\cM}$ be the map associated with $\bar g$ and $\bar h^{(i)}$. For any initial conditions $f$ and $\bar{f}$ in $\cC^{-1+4\kappa}_{\e_{\ell}}$, both $\cM_{T,f}$ and $\bar{\cM}_{T,\bar{f}}$ admit a unique fixed point $v$ and $\bar v$. Furthermore, we have
\begin{equs}
v_t - \bar{v}_t &= \Big(\cM_{T,f} v-\cM_{T,f} \bar v\Big)_t + \int_0^t P_{t-s} * \Big(\bar{v}_s(g-\bar{g}) + D_{x_i} \bar{v}_s(h^{(i)}-\bar{h}^{(i)})\Big)ds\\
&\;\;\;+ P_t*(f-\bar{f}) \;.
\end{equs}
Using (\ref{Eq:FixedPt}) and (\ref{Eq:Contraction}), we deduce that
\begin{equs}
\$ v - \bar{v} \$_{\ell,T} &\lesssim T^{\frac{1}{2}-2\kappa} \$ v - \bar{v} \$_{\ell,T} \big( \|g\|_{-\kappa,\pa} + \|\bar g\|_{-\kappa,\pa} + \|h^{(i)}\|_{-\kappa,\pa} + \|\bar{h}^{(i)}\|_{-\kappa,\pa}\big)\\
\ &+ T^{\frac{1}{2}-2\kappa} \$ v\$_{\ell,T} \big(\|\bar g -g\|_{-\kappa,\pa} + \|\bar{h}^{(i)} - h^{(i)} \|_{-\kappa,\pa}\big)\\
\ &+ \| f - \bar{f} \|_{-1+4\kappa,\ell} \;,
\end{equs}
uniformly over all $\ell$ in a compact set of $\R$ and all $T$ in a compact set of $\R_+$. Fix $R>0$. There exists $T>0$ such that
\[ \$ v - \bar{v} \$_{\ell,T} \lesssim \| f - \bar{f} \|_{-1+4\kappa,\ell} + T^{\frac{1}{2}-2\kappa} \big(\|\bar g -g\|_{-\kappa,\pa} + \|\bar{h}^{(i)} - h^{(i)} \|_{-\kappa,\pa}\big) \;,\]
uniformly over all $\ell$ in a compact set of $\R$ and all $g,\bar g, h, \bar h$ such that $\$v\$_{\ell,T}$, $\|g\|_{-\kappa,\pa}$, $\|\bar g\|_{-\kappa,\pa}$, $\|h^{(i)}\|_{-\kappa,\pa}$ and $\|\bar{h}^{(i)}\|_{-\kappa,\pa}$ are smaller than $R$. This yields the continuity of the solution map on $(0,T]$. By iterating the argument as above, we obtain continuity on any bounded interval.
\end{proof}
We are now in position to prove the main result of this article.

\begin{proof}[of Theorem~\ref{ThMainPrecise}] Let $u_0$ be an element in $\cC^{-1+4\kappa}_{\e_\ell}$ for a given $\ell \in \R$. Let $f_\varepsilon:= u_0 e^{Y_\varepsilon}$. By Corollary \ref{Cor:Y} and Theorem \ref{Th:Young}, $f_\varepsilon$ converges to $f=u_0 e^Y$ in $\cC^{-1+4\kappa}_{\e_{\ell''}}$ for any $\ell'' > \ell$. Let $v_\varepsilon$ be the unique fixed point of $\cM_{T,f_\varepsilon}$ with $g_\varepsilon= Z_\varepsilon - F*\xi_\varepsilon$ and $h_\varepsilon^{(i)} := -2 D_{x_i}Y_\varepsilon$. By Corollary \ref{Cor:Y} and Proposition \ref{Prop:Z}, we know that $g_\varepsilon$, $h_\varepsilon^{(i)}$ converge in probability to
\[ g = Z-F*\xi \;,\qquad h^{(i)} = -2D_{x_i} Y \;,\]
in $\cC^{-\kappa}_{\pa}$. Notice that the convergence of $F*\xi_\varepsilon$ towards $F*\xi$ is a consequence of Lemma~\ref{Lemma:xi}, since $F$ is a compactly supported, smooth function. Therefore, Proposition~\ref{Prop:FixedPt} ensures that $v_\varepsilon$ converges in probability in $\cE_{\ell'',T}^{1+2\kappa}$ to the unique fixed point $v$ of the map $\cM_{T,f}$ associated to $g, h^{(1)}, h^{(2)}$. Moreover, Theorem \ref{Th:Young} ensures that, for any $\ell' > \ell''$, $u_\varepsilon = v_\varepsilon e^{-Y_\varepsilon}$ converges to $u=v e^{-Y}$ in the space $\cE_{\ell',T}^{1-\kappa}$.
\end{proof}

\bibliographystyle{Martin}
\bibliography{library_PAM_R2}

\end{document}